\documentclass[11pt]{amsart}
\usepackage{amssymb,color}
\usepackage{amsmath}
\usepackage{latexsym}
\usepackage{verbatim}
\usepackage[left=1.3in,right=1.3in,top=1in,bottom=1in]{geometry}

\def\e{{\varepsilon}}

\newcommand{\Z}{{\mathbb Z}}
\newcommand{\R}{{\mathbb R}}

\newcommand{\D}{{\mathbb D}}
\newcommand{\T}{{\mathbb T}}

\newcommand{\pD}{{\partial\mathbb{D}}}

\newcommand{\CC}{{\mathcal C}}

\newcommand{\CE}{{\mathcal E}}

\newcommand{\CL}{{\mathcal L}}
\newcommand{\CM}{{\mathcal M}}

\renewcommand{\Re}{{\mathrm{Re}}}
\renewcommand{\Im}{{\mathrm{Im}}}

\newtheorem{theorem}{Theorem}[section]
\newtheorem{lemma}[theorem]{Lemma}
\newtheorem{prop}[theorem]{Proposition}

\newtheorem{coro}[theorem]{Corollary}

\theoremstyle{definition}
\newtheorem*{remark}{Remark}
\theoremstyle{definition}
\newtheorem{defi}[theorem]{Definition}

\numberwithin{theorem}{section}
\numberwithin{equation}{section}

\begin{document}

\author[F.\ Wang]{Fengpeng Wang}
\address{Ocean University of China, Qingdao 266100, Shandong, China and Rice University, Houston, TX~77005, USA}
\email{wfpouc@gmail.com}
\thanks{F.W.\ was supported by CSC (No.201606330003) and NSFC (No.11571327).}

\author[D.\ Damanik]{David Damanik}
\address{Rice University, Houston, TX~77005, USA}
\email{damanik@rice.edu}
\thanks{D.D.\ was supported in part by NSF grants DMS--1361625 and DMS--1700131.}

\title{Anderson Localization for Quasi-Periodic CMV Matrices and Quantum Walks}

\begin{abstract}
We consider CMV matrices, both standard and extended, with analytic quasi-periodic Verblunsky coefficients and prove Anderson localization in the regime of positive Lyapunov exponents. This establishes the CMV analog of a result Bourgain and Goldstein proved for discrete one-dimensional Schr\"odinger operators. We also prove a similar result for quantum walks on the integer lattice with suitable analytic quasi-periodic coins.
\end{abstract}

\maketitle

\section{Introduction and Preliminaries}

The spectral theory of Schr\"odinger operators with random or quasi-periodic potentials modeling Hamiltonians of quantum mechanical systems has been an area of very active study since the 1970's. These two classes of potentials can be generated dynamically. This makes a unified proof of basic spectral results possible, such as the almost sure constancy of the spectrum and spectral type, including absolutely continuous spectrum, singular continuous spectrum and point spectrum. On physical grounds, the main motivation for studying spectral properties is their close connection (via, e.g., the RAGE Theorem) with dynamical properties of the corresponding time evolution in quantum mechanics given by the Schr\"odinger equation.

In physics, there is an interesting phenomenon called Anderson localization \cite{anderson}, which describes insulating behavior in the sense that quantum states are essentially localized in a suitable bounded region for all times, and which is named after the winner of the 1977 Nobel Prize in Physics, P.~W.~Anderson. In mathematics, Anderson localization often means that the corresponding operator has pure point spectrum with exponentially decaying eigenfunctions.

Recent developments in the theory of orthogonal polynomials on the unit circle (OPUC) have emphasized the importance of a class of unitary matrices called CMV matrices, a special class of unitary five-diagonal matrices. The corresponding unitary semi-infinite five-diagonal matrices were studied by Cantero, Moral and Vel$\acute{\rm a}$zquez \cite{cmv05}, while in \cite{OPUC1}, Simon introduced the corresponding notion of unitary doubly infinite five-diagonal matrices and coined the term extended CMV matrices.

In this paper, we study the spectral theory of regular and extended CMV matrices. Since these matrices can be regarded as the unitary analogues of Schr\"odinger (or more generally Jacobi) matrices, one expects many analogies between the spectral theory of CMV matrices and that of discrete Schr\"odinger operators. This has been worked out in a variety of settings, many of them were already addressed in \cite{OPUC2}, including periodic and decaying coefficients. See also the recent paper \cite{randomcmv} devoted to CMV matrices with random coefficients.

Conspicuously absent, however, is a systematic treatment of the case of analytic quasi-periodic coefficients. The latter have been studied extensively in the Schr\"odinger setting, but the corresponding CMV analogues do not yet exist in the same generality. Based on the existing Schr\"odinger literature, one expects purely absolutely continuous spectrum at small coupling and, for typical frequencies, Anderson localization in the region of spectral parameters with positive Lyapunov exponent.

Simon discusses this in the remarks and historical notes section following the general discussion of dynamically defined Verblunksy coefficients; see \cite[pp.706--707]{OPUC2}. He points out that ``understanding almost periodic Verblunsky coefficients is an intriguing open area'' and in particular describes what one should expect in the small coupling regime, as well as the difficulties one faces when trying to establish a localization result for quasi-periodic Verblunsky coefficients.

In this paper we address the questions concerning localization and prove the desired results for CMV matrices with analytic quasi-periodic coefficients in the positive Lyapunov exponent regime. In a forthcoming companion paper we will address the other question and establish purely absolutely continuous spectrum at small coupling.


\subsection{CMV matrices}

CMV matrices can be thought of in terms of being unitary matrices in canonical form or in terms of being a finite-band matrix representation associated with OPUC. Let us start with the OPUC point of view.

A probability measure $\mu$ on the unit circle $\partial\mathbb{D} = \{ z \in \mathbb{C} : |z| = 1 \}$ is called \emph{nontrivial} if it is not supported on a finite set. Let $\Phi_{n}(z)$ be the monic orthogonal polynomials, that is,
\begin{equation*}
    \text{$\Phi_{n}(z) = P_{n}[z^{n}]$ \quad  $P_{n} \equiv$ projection onto $\{ 1, z, \dots, z^{n-1} \}$ in $L^2(\partial\mathbb{D}, d\mu)$.}
\end{equation*}
The orthonormal polynomials are
\begin{equation*}
    \varphi_{n}(z) = \frac{\Phi_{n}(z)}{\lVert \Phi_{n}(z) \rVert_{\mu}},
\end{equation*}
where $\lVert \cdot \rVert_{\mu}$ is the norm of $L^2(\partial\mathbb{D}, d\mu)$.

For any polynomial $Q_{n}(z)$ of degree $n$, we define the \emph{reversed polynomial} (also called \emph{Szeg\H o dual} \cite{cmvfive}) $Q^{*}_{n}(z)$ by
\begin{equation*}
    Q^{*}_{n}(z)=z^{n} \overline{Q_{n}(1/\bar{z})}.
\end{equation*}
More specifically,
\begin{equation*}
    Q_{n}(z) = \sum_{j=0}^{n} c_{j}z^{j} \Rightarrow  Q^{*}_{n}(z) = \sum_{j=0}^{n} \bar{c}_{n-j}z^{j},
\end{equation*}
so the order of coefficients is reversed and complex conjugates are taken. Note that the symbol $^{*}$ does depend on $n$, but this dependence is often left implicit.

The monic orthogonal polynomials obey the \emph{Szeg\H o recurrence}, given by
\begin{equation*}
    \Phi_{n+1}(z) = z \Phi_{n}(z) - \bar{\alpha}_{n} \Phi^{*}_{n}(z),
\end{equation*}
with suitably chosen parameters $\alpha_{0}, \alpha_{1}, \cdots$, which are called the \emph{Verblunsky coefficients} and which belong to $\mathbb{D} = \{ z \in \mathbb{C}: \lvert z \rvert < 1 \}$, that is, $|\alpha_{j}| < 1$.

By orthonormalizing $1, z, z^{-1}, z^{2}, z^{-2}, \dots$, we get the \emph{CMV basis} $\{ \chi_{j} \}_{j=0}^{\infty}$ of $L^2(\partial\mathbb{D}, d\mu)$, and the matrix representation of multiplication by $z$ relative to the CMV basis gives rise to the \emph{CMV matrix} $\mathcal{C}$,
\begin{equation*}
    \mathcal{C}_{ij} = \langle \chi_{i}, z\chi_{j} \rangle.
\end{equation*}
It turns out that $\mathcal{C}$ takes the form
\begin{equation*}
    \mathcal{C}=
    \begin{bmatrix}
         \bar{\alpha}_{0} &  \bar{\alpha}_{1} \rho_{0}     &  \rho_{1} \rho_{0}               &                 &             &             &   \\
         \rho_{0}         &  -\bar{\alpha}_{1} \alpha_{0}  & -\rho_{1} \alpha_{0}             &                 &                &      &   \\
                          &  \bar{\alpha}_{2} \rho_{1}     & -\bar{\alpha}_{2} \alpha_{1}     &  \bar{\alpha}_{3} \rho_{2} & \rho_{3} \rho_{2} &        &   \\
                          &  \rho_{2} \rho_{1}             & -\rho_{2} \alpha_{1}             & -\bar{\alpha}_{3} \alpha_{2}  &  -\rho_{3} \alpha_{2}&     &  \\
                   &       &        &  \bar{\alpha}_{4} \rho_{3}  &  -\bar{\alpha}_{4} \alpha_{3}  & \bar{\alpha}_{5}\rho_{4} &     \\
                   &       &        &  \rho_{4}\rho_{3} & -\rho_{4}\alpha_{3}  & -\bar{\alpha}_{5}\alpha_{4}   &     \\
                     &    &   &    & \ddots  & \ddots  & \ddots
    \end{bmatrix}
\end{equation*}
where $\rho_j = \sqrt{1 - |\alpha_j|^2}$, $j \ge 0$.

There is a natural way to define a two-sided analog by simply extending the five diagonals in the obvious way using suitable $\alpha_j \in \D$, $j \le -1$, and again setting $\rho_j = \sqrt{1 - |\alpha_j|^2}$, $j \le -1$. This yields the extended CMV matrix $\mathcal{E}$, given by
\begin{equation*}
    \mathcal{E}=
    \begin{bmatrix}
         \ddots & \ddots & \ddots &     &   &    &   &     \\
                & -\bar{\alpha}_{0}\alpha_{-1}  & \bar{\alpha}_{1}\rho_{0}  & \rho_{1}\rho_{0} &  &   &   &    \\
                & -\rho_{0}\alpha_{-1}  & -\bar{\alpha}_{1}\alpha_{0}  &  -\rho_{1}\alpha_{0} &  &   &    &      \\
                &   & \bar{\alpha}_{2}\rho_{1} & -\bar{\alpha}_{2}\alpha_{1}  & \bar{\alpha}_{3}\rho_{2}  & \rho_{3}\rho_{2} &       &           \\
                &   & \rho_{2}\rho_{1} & -\rho_{2}\alpha_{1}  & -\bar{\alpha}_{3}\alpha_{2}  &  -\rho_{3}\alpha_{2} &       &               \\
                &   &   &    & \bar{\alpha}_{4}\rho_{3} & -\bar{\alpha}_{4}\alpha_{3}  & \bar{\alpha}_{5}\rho_{4}  &         \\
                &   &   &    & \rho_{4}\rho_{3} & -\rho_{4}\alpha_{3}  & -\bar{\alpha}_{5}\alpha_{4}  &                \\
                &   &   &    &    &  \ddots & \ddots & \ddots
    \end{bmatrix}.
\end{equation*}
To emphasize that $\mathcal{C}$ acts on the right half line, it is often also called a \emph{half-line CMV matrix}.

In $\mathcal{C}$ and $\mathcal{E}$, all unspecified matrix entries are implicitly assumed to be zero. It is known that the half-line (resp., extended) CMV matrix is a unitary operator on $\ell^{2}(\mathbb{N})$ (resp., $\ell^{2}(\mathbb{Z})$).

\subsection{Szeg\H o cocycle and Lyapunov exponent.}

In this paper, we consider Verblunsky coefficients that are generated dynamically, that is, $\alpha_{n}(x) = \alpha(T^{n}x) = \alpha(x+n\omega)$, where $T$ is an invertible map of the form $Tx=x+\omega$, and $x, \omega \in \T := \R/\Z$ are called \emph{phase} and \emph{frequency} respectively. We assume that the \emph{sampling function} $\alpha(x): \T \to \mathbb{D}$ is analytic and satisfies
\begin{equation}\label{e.sampfunctassump}
    \int_{\T} \log (1-|\alpha(x)|) d\mu > -\infty.
\end{equation}
Note that for rational frequency $\omega$, $\alpha_{n}(x)$ is periodic for every phase $x$. When the frequency $\omega$ is irrational, the sequence is quasi-periodic.

The Szeg\H o recurrence is equivalent to the following one,
\begin{equation}{\label{szegorecursion}}
    \rho_{n}(x) \varphi_{n+1}(z) = z \varphi_{n}(z) - \bar{\alpha}_{n}(x) \varphi_{n}^{*}(z),
\end{equation}
where $\rho_{n}(x)= \rho(x+n\omega)$ and $\rho(x) = (1-|\alpha(x)|^{2} )^{1/2}$.

Applying $^{*}$ to both sides of \eqref{szegorecursion}, we get
\begin{equation}{\label{szegorecursion2}}
    \rho_{n}(x) \varphi_{n+1}^{*}(z) =  \varphi_{n}^{*}(z) - \alpha_{n}(x) z\varphi_{n}(z).
\end{equation}
Then \eqref{szegorecursion} and \eqref{szegorecursion2} can be written as
\begin{equation*}
    \begin{bmatrix}
        \varphi_{n+1} \\
        \varphi_{n+1}^{*}
    \end{bmatrix}= S^{z}(x + n\omega)
    \begin{bmatrix}
        \varphi_{n} \\
        \varphi_{n}^{*}
    \end{bmatrix},
\end{equation*}
where
\begin{equation*}
    S^{z}(x)=\frac{1}{\rho_{n}(x)}\begin{bmatrix}
              z     &   -\bar{\alpha}_{n}(x) \\
              -\alpha_{n}(x) z & 1
    \end{bmatrix}.
\end{equation*}

Since $\det S^{z}(x) = z$, some authors prefer to study the following determinant 1 matrix instead,
\begin{equation*}
   M^{z}(x)=\frac{1}{\rho(x)}\begin{bmatrix}
              \sqrt{z}               &   -\frac{\bar{\alpha}(x)}{\sqrt{z}} \\
              -\alpha(x) \sqrt{z}   &     \frac{1}{\sqrt{z}}
    \end{bmatrix} \in \mathbb{SU}(1, 1),
\end{equation*}
which is called the \emph{Szeg\H o cocycle map}.

Then the $n$-step transfer matrix is defined by
\begin{equation}{\label{nstep}}
    M^{z}_{n}(x) = \prod \limits_{j=n-1}^{0} M^{z}(x+j\omega).
\end{equation}
It follows from \eqref{nstep} that
\begin{equation*}
     M^{z}_{n_{1} + n_{2}}(x) = M^{z}_{n_{2}}(x+n_{1}\omega)  M^{z}_{n_{1}}( x),
\end{equation*}
and hence
\begin{equation}{\label{mn1n2}}
    \log \| M^{z}_{n_{1}+n_{2}}( x) \| \leq \log \| M^{z}_{n_{1}}(x) \| + \log \| M^{z}_{n_{2}}(x+n_{1}\omega) \|.
\end{equation}
By integration in $x$, we get
\begin{equation*}
    L_{n_{1}+n_{2}}(z) \leq \frac{n_{1}}{n_{1}+n_{2}}L_{n_{1}}(z) + \frac{n_{2}}{n_{1}+n_{2}}L_{n_{2}}(z),
\end{equation*}
where
\begin{equation*}
    L_{n}(z) = \frac{1}{n} \int_{\T} \log \| M^{z}_{n}(x) \| dx.
\end{equation*}
This implies
\begin{equation*}
   \text{$L_{n}(z) \leq L_{m}(z)$ if $ m<n, m | n$}
\end{equation*}
and
\begin{equation*}
  \text{$L_{n}(z) \leq L_{m}(z) + C\frac{m}{n}$    if $m<n$.}
\end{equation*}
Notice that \eqref{mn1n2} means $\log \lVert M^{z}_{n}(x) \rVert$ is subadditive. Therefore, the \emph{Lyapunov exponent} $L(z)$ defined by
\begin{equation*}
    L(z) = \lim_{n \to \infty} L_{n}(z)
\end{equation*}
exists.

\subsection{Main Theorem}

From the discussion above, given an analytic function $\alpha(x): \T \to \mathbb{D}$ and the invertible map $Tx=x+\omega$, we can define a sequence of Verblunsky coefficients and the corresponding CMV matrix $\mathcal{E}_{\omega}(x)$ dynamically.

The first main theorem in our paper is:

\begin{theorem}{\label{talcmv}}
Consider the family of extended CMV matrices $\{ \mathcal{E}_{\omega}(x)\}$, where the Verblunsky coefficients are given by $\{ \alpha_{n}(x) = \alpha(x+n\omega)\}_{n \in \mathbb{Z}}$ and $\alpha(x): \T \to \mathbb{D}$ is analytic and obeys \eqref{e.sampfunctassump}.

Let $\mathcal{I} \subset \T$ and $\mathcal{K} \subset \partial \mathbb{D}$ be compact intervals. Assume that the Lyapunov exponent obeys
\begin{equation}
 L(z) > \delta_{0} > 0
\end{equation}
for every $\omega \in \mathcal{I}$ and every $z \in \mathcal{K}$.

Fix $x_{0} \in \T$. Then for almost every $\omega \in \mathcal{I}$, $\mathcal{E}_{\omega}(x_{0})$ has pure point spectrum on $\mathcal{K}$ with exponentially decaying eigenfunctions {\rm (}i.e. Anderson localization{\rm )}.
\end{theorem}

\begin{remark}
It is well known (see, e.g., \cite[Theorem~3.4]{DFLY16}) that if every generalized eigenfunction $\xi = (\xi_{n})_{n \in \mathbb{Z}}$ of $\CE_{\omega}$ decays exponentially, then the operator $\mathcal{E}_{\omega}$ displays Anderson localization. That is, we need to prove for any $\xi = (\xi_{n})_{n \in \mathbb{Z}}$ and $z \in \partial\mathbb{D}$ satisfying
\begin{equation*}
    |\xi _{n}| \lesssim |n|^{C}
\end{equation*}
($\lesssim$ denotes inequality up to a multiplicative constant) and
$$
\mathcal{E}_{\omega} \xi = z \xi,
$$
that
\begin{equation*}
    \text{$\lvert \xi _{n} \rvert  \lesssim e^{-c\lvert n \rvert}$  for some $c > 0$.}
\end{equation*}
\end{remark}

It is not hard to see that same result can be established for half-line CMV matrices:

\begin{theorem}{\label{halflineal}}
Consider the family of CMV matrices $\{ \mathcal{C}_{\omega}(x)\}$, where the Verblunsky coefficients are given by $\{ \alpha_{n}(x) = \alpha(x+n\omega)\}_{n=0}^{\infty}$ and $\alpha(x): \T \to \mathbb{D}$ is analytic.

Let $\mathcal{I} \subset \T$ and $\mathcal{K} \subset \partial \mathbb{D}$ be compact intervals. Assume that the Lyapunov exponent obeys
\begin{equation*}
L(z) > \delta_{0} > 0
\end{equation*}
for every $\omega \in \mathcal{I}$ and every $z \in \mathcal{K}$.

Fix $x_{0} \in \T$. Then for almost every $\omega \in \mathcal{I}$, $\mathcal{C}_{\omega}(x_{0})$ has pure point spectrum on $\mathcal{K}$ with exponentially decaying eigenfunctions {\rm (}i.e. Anderson localization{\rm )}.
\end{theorem}

To make our Theorem~\ref{talcmv} meaningful, we apply it to a specific model which was introduced in \cite{zhangpositive} by Zhenghe Zhang.

Consider Verblunsky coefficients with a coupling constant $\lambda$, that is, $\alpha(x)=\lambda v(x)$ with $\lambda \in (0, 1)$, and $v(x) = e^{2\pi i h(x)}$, where $h : \T \to \T$ is a non-constant analytic function. Notice that such a function $h \in C^{\omega}(\T, \T)$ can be written as $h(x) = kx + \theta(x)$ with $\theta \in C^{\omega}(\T, \mathbb{R})$ and the degree $k \in \Z$. Note also that there is a largest positive integer $q=q(\theta)$ such that $\theta(x+\frac{1}{q}) = \theta(x)$.

Recall the uniform positivity of its Lyapunov exponent (i.e., \cite[Theorem A]{zhangpositive}):

\begin{prop}{\label{pzz1}}
Let $\alpha(x)$, $\theta(x)$, $k$ be as above and let $\pi _{1}: \T \times \partial\mathbb{D} \to \T$ be the projection to the first component. Then there exists a finite set $\mathcal{F} = \mathcal{F}(\theta, k) \subset \T \times \partial\mathbb{D}$ with $\pi_{1}(\mathcal{F}) \subset \{ \frac{p}{q}: p=0,1,\cdots,q-1 \}$ with the following property:

For any compact set $\mathcal{A} \subset (\T \times \partial\mathbb{D}) \backslash \mathcal{F}$, there exists a constant $c_{0} = c_{0}(\theta, k, \mathcal{A}) \in \mathbb{R}$ such that
\begin{equation*}
L(z) \geq - \frac{1}{2} \log(1-\lambda) + c_{0}
\end{equation*}
for all $(\omega, z;\lambda) \in \mathcal{A} \times (0,1)$.
\end{prop}

From Proposition~\ref{pzz1} we may infer that $L(z)$ is positive if $\lambda$ is sufficiently close to 1, and hence Theorem~\ref{talcmv} implies the following:

\begin{coro}{\label{example}}
Assume the setting of Proposition~\ref{pzz1} and fix $x_{0} \in \T$. Then there exists $\lambda_{0} \in (0,1)$ such that for every $\lambda \in (\lambda_{0}, 1)$ and almost every $\omega$, $\mathcal{E}_{\omega}(x_{0})$ displays Anderson localization in $\{ z \in \partial \D : (\omega, z) \in \mathcal{A} \}$.
\end{coro}

Moreover, \cite[Corollary 4]{zhangpositive} implies the following result:
\begin{coro}
Assume the setting of Proposition~\ref{pzz1} and fix $x_{0} \in \T$. Then for almost every $\omega$ with $d(\omega, \pi_{1}(\mathcal{F})) > 0$, there exists $\lambda_{0} \in (0,1)$ such that $\mathcal{E}_{\omega}(x_{0})$ exhibits Anderson localization on $\partial\mathbb{D}$ for all $\lambda \in (\lambda_{0}, 1)$.
\end{coro}

\section{A Large Deviation Estimate}

In this paper we use the method developed by Jean Bourgain and Michael Goldstein in \cite{BG00} to prove Anderson localization, which means that a large deviation estimate (LDT) for analytic quasi-periodic Szeg\H o cocycles will play an important role.

The first LDT for quasi-periodic Schr\"odinger operators was obtained in \cite{BG00} and used there to prove Anderson localization. In this section, which is based on the more recent \cite[Theorem 1]{YZ14}, we will state an LDT for analytic quasi-periodic Szeg\H o cocycles that is stronger than that in \cite{BG00} and hence sufficient to prove Anderson localization.

Fix an irrational $\omega \in \T$ and consider the continued fraction expansion $\omega = [a_{1}, a_{2}, \cdots]$ with convergents $p_{s}/q_{s}$ for $s \in \mathbb{N}$. Let
\begin{equation*}
    \beta(\omega) = \limsup_{s \to \infty} \frac{\log q_{s+1}}{q_{s}}.
\end{equation*}

Since $\alpha(x)$ is analytic on $\T$, it has a bounded extension to a complex strip $\{ x+iy : \lvert y \rvert < h \}$, $h > 0$. Define its norm by
\begin{equation*}
    \lVert \alpha \rVert_{h} = \sup_{\lvert y \rvert < h} \lvert \alpha(x+iy) \rvert,
\end{equation*}
and set
\begin{equation}\label{e.calphadef}
C_{\alpha} = \log \frac{2}{\sqrt{1- \lVert \alpha \rVert_{h}^{2}}}.
\end{equation}

\begin{lemma}{\label{lldt}}
For any $\kappa > 0$, there exist $N(\kappa, C_{\alpha})$ and absolute constants $c_{0}, c_{1}$ that are independent of $\kappa, C_{\alpha}$, such that if $\beta(\omega) < c_{0} \cdot \kappa / C_{\alpha}$, $z \in \partial\mathbb{D}$ and $n > N(\kappa, C_{\alpha})$, then
\begin{equation}{\label{ildtszego}}
{\rm mes} \left\{ x \in \T : \Big | \frac{1}{n} \log \lVert M^{z}_{n}(x) \rVert -L_{n}(z) \Big |  > \kappa \right\} < e^{-(c_{1}/C_{\alpha}^{3})\kappa ^{3}n}.
\end{equation}
\end{lemma}

\begin{remark}
The condition $\beta(\omega) < c_{0} \cdot \kappa / C_{\alpha}$ is trivially satisfied in the case $\beta(\omega) = 0$. It is well known that the Diophantine numbers form a subset of $\{ \omega : \beta(\omega)=0 \}$ (cf.~Definition~\ref{def.DC} below), and hence our version is stronger than \cite[Lemma 1.1]{BG00}. In addition, the smallness of the measure in our version does not depend on $\sigma \in (0,1)$.
\end{remark}

\begin{proof}
It is easy to check that $M^{z}(x)$ is conjugate to an $\mathrm{SL}(2, \mathbb{R})$ matrix $A^z(x)$ via
\begin{equation*}
    Q = \frac{-1}{1+i} \begin{bmatrix}
              1   &  -i\\
              1   &   i
    \end{bmatrix} \in \mathbb{U}(2),
\end{equation*}
that is,
\begin{equation}{\label{esl2r}}
        A^{z}(x) = Q^{*}M^{z}(x)Q \in {\rm SL(2, \mathbb{R})}.
\end{equation}
Since $M^{z}(x)$ is analytic in $x \in \T$ and $z \in \partial\mathbb{D}$, so is $A^{z}(x)$.
By conjugacy, $\lVert A^{z}(x) \rVert = \lVert M^{z}(x) \rVert$ and $\lVert A^{z}_{n}(x) \rVert = \lVert M^{z}_{n}(x) \rVert $, where
\begin{equation}{\label{tsl2r}}
    A^{z}_{n}(x)  = \prod \limits_{j=n-1}^{0} A^{z}(x + j\omega).
\end{equation}
Hence \eqref{ildtszego} is equivalent to
\begin{equation}{\label{ildtszego2}}
{\rm mes} \left\{ x \in \T : \Big | \frac{1}{n} \log \lVert A^{z}_{n}(x) \rVert -\tilde{L}_{n}(z) \Big |  > \kappa \right\} < e^{-(c_{1}/C_{\alpha}^{3})\kappa ^{3}n},
\end{equation}
where
\begin{equation*}
    \tilde{L}_{n}(z) = \frac{1}{n} \int_{\T} \log \lVert A^{z}_{n}(x) \rVert dx.
\end{equation*}
According to the polar decomposition of $\text{SL}(2,\mathbb{C})$ matrices (for details, see \cite[Section~3.2]{zhangpositive} or Section \ref{polardecomposition} in this paper),
\begin{equation*}
    tr(A^{*}A)=|a|^{2}+|b|^{2}+|c|^{2}+|d|^{2}=\|A\|^{2}+\|A\|^{-2},
\end{equation*}
where $a,b,c,d$ are the entries of matrix $A$.

Then a direct calculation shows
\begin{equation*}
    \begin{split}
    \lVert A^{z}(x) \rVert + \lVert A^{z}(x) \rVert ^{-1}
    &= \lVert M^{z}(x) \rVert + \lVert M^{z}(x) \rVert ^{-1} \\
    &= \sqrt{tr((M^{z}(x))^{*}M^{z}(x))+2} \\
    &= \frac{\sqrt{2+2|\alpha(x)|^{2}+2\rho^{2}(x)}}{\rho(x)} \\
    &= \frac{2}{\sqrt{1- \lvert \alpha(x) \rvert^{2}}}.
    \end{split}
\end{equation*}
Hence, $\sup_{\lvert y \rvert < h} \lVert A^{z}(x+iy) \rVert, \sup_{\lvert y \rvert < h} \lVert A^{z}(x+iy)^{-1} \rVert \leq \frac{2}{\sqrt{1- \lVert \alpha \rVert_{h}^{2}}}$ and
\begin{equation*}
     \sup_{\lvert y \rvert < h}  \frac{1}{n}\log \lVert A^{z}_{n}(x+iy)\rVert
         \leq \log \frac{2}{\sqrt{1- \lVert \alpha \rVert_{h}^{2}}} = C_{\alpha}.
\end{equation*}
Since $A^{z}(x) \in {\rm SL(2, \mathbb{R})}$ is analytic, the proof of \eqref{ildtszego2} is the same as the proof of \cite[Theorem 1]{YZ14}, we just need to replace the use of $C_{v}$ there by our $C_{\alpha}$.
\end{proof}

\section{Estimating Transfer Matrices and Green's Functions}\label{section3}

In this section we give various bounds on the norms of $n$-step transfer matrices. First we establish a relation between $n$-step transfer matrices and the corresponding Lyapunov exponents. Then, by the method developed by Helge Kr\"uger in \cite{Krueg2013IMRN}, we can get that the entries of the Green's function decay exponentially off the diagonal, which is essential for proving Anderson localization.

\subsection{Norm Estimates of Transfer Matrices by Lyapunov Exponents}{\label{aetm}}

In this subsection we state two lemmas whose Schr\"odinger version was proved in \cite{BG00}. Since the proofs in our case are somewhat similar to the proofs of the corresponding statements in \cite{BG00}, our discussion focuses primarily on the difference between the proofs. From the proof of Lemma~\ref{lldt}, we know that $\| M^{z}_{n}(x) \|=\|A^{z}_{n}(x)\|$, hence we can use them interchangeably.

\begin{defi}\label{def.DC}
We say that $\omega \in \T$ satisfies a Diophantine condition, denoted $\omega \in \mathrm{DC}$, if for some $A > 0$, we have
\begin{equation}\label{dcac2}
\lVert k \omega \rVert > c \lvert k \rvert^{-A} \text{ for } k \in \mathbb{Z} \backslash \{0\}.
\end{equation}
\end{defi}

It is well known that Lebesgue almost every $\omega \in \T$ satisfies a Diophantine condition. Moreover, $\omega \in \mathrm{DC}$ implies $\beta(\omega) = 0$.

\begin{lemma}\label{lub}
Assume $\omega \in \mathrm{DC}$. Then for all $x \in \T$ and $z \in \partial\mathbb{D}$, we have
\begin{equation}{\label{iub}}
\frac{1}{n} \log \lVert M^{z}_{n}(x) \rVert < L_{n}(z) + C\kappa
\end{equation}
for any $\kappa > 0$ and $n$ large enough.
\end{lemma}

\begin{proof}
Since $A^{z}_{n}(x)$ is analytic, it is easy to check that $\| A^{z}_{n}(x+iy) \| + \| (A^{z}_{n}(x+iy))^{-1} \| \lesssim C^{n}$. From the proof of \cite[Lemma~2.1]{BG00}, it is clear that
\begin{equation*}
\frac{1}{n} \log \lVert A^{z}_{n}(x) \rVert < \tilde{L}_{n}(z) + C\kappa
\end{equation*}
for $n$ large enough, which is equivalent to \eqref{iub}.
\end{proof}

\begin{lemma}{\label{lar}}
Assume $\omega \in \T$ satisfies \eqref{dcac2}. Then, for
\begin{equation*}
J > n^{2A},
\end{equation*}
we have
\begin{equation}{\label{ear}}
\frac{1}{J} \sum_{j=1}^{J} \Big (\frac{1}{n} \log \lVert M^{z}_{n}(x+ j\omega) \rVert\Big ) = L_{n}(z) + O \Big (\frac{1}{n}\Big),
\end{equation}
uniformly for all $x$ and $z$.
\end{lemma}

\begin{proof}
It follows from \eqref{esl2r} and \eqref{tsl2r} that $\lVert \partial_{x}A^{z}_{n}(x) \rVert < C^{n}$ and $\Big \lvert \partial_{x} [\frac{1}{n} \log \lVert A^{z}_{n}(x) \rVert] \Big \rvert < C^{n}$, and hence \cite[Lemma~3.1]{BG00} implies
\begin{equation*}
\frac{1}{J} \sum_{j=1}^{J} \Big (\frac{1}{n} \log \lVert A^{z}_{n}(x+ j\omega) \rVert\Big ) = \tilde{L}_{n}(z) + O\Big(\frac{1}{n}\Big),
\end{equation*}
which is equivalent to \eqref{ear}.
\end{proof}

\subsection{Factorization and Restriction of CMV Matrices}

Next, we adopt some definitions and notations from \cite{Krueg2013IMRN, OPUC1}, which we will use below to get the desired exponential decay property.

Let $\{\alpha_{n}\}_{n=0}^{\infty}$ be the sequence of Verblunsky coefficients of a half-line CMV matix $\CC_{\omega}$. Define the unitary matrices
\begin{equation*}
    \Theta_{n}=
\begin{bmatrix} \bar{\alpha}_{n} & \rho_{n} \\ \rho_{n} & - \alpha_{n}
\end{bmatrix}.
\end{equation*}
Denote $\CL_{+}$, $\CM_{+}$ by
\begin{equation*}
\CL_{+}=\begin{bmatrix}
\Theta_{0} &  &  \\   & \Theta_{2} & \\  &    &  \ddots
\end{bmatrix}, \quad
\CM_{+}=\begin{bmatrix}
1 &  &  \\   & \Theta_{1} & \\  &    &  \ddots
\end{bmatrix},
\end{equation*}
where $1$ represents the $1 \times 1$ identity matrix. It is well known that $\CC_{\omega} = \CL_{+} \CM_{+}$.

Now let $\{\alpha_{n}\}_{n=-\infty}^{\infty}$ be a bi-infinite sequence of Verblunsky coefficients and denote the corresponding extended CMV matrices by $\CE_{\omega}$ and define $\CL_{\omega}, \CM_{\omega}$ as
\[
\CL_\omega
=
\bigoplus_{j\in\Z}\Theta_{2j},
\quad
\CM_\omega
=
\bigoplus_{j\in\Z} \Theta_{2j+1}.
\]
Then the analogous factorization of $\CE_{\omega}$ is given by $\CE_{\omega} = \CL_{\omega} \CM_{\omega}$.

Let $\CE_{\omega, [a,b]}$ denote the restriction of an extended CMV matrix to the finite interval $[a,b]$, defined by
\begin{equation*}
\CE_{\omega,[a,b]} = P_{[a,b]} \CE_{\omega} (P_{[a,b]})^{*},
\end{equation*}
where $P_{[a,b]}$ is the projection $\ell^{2}(\Z) \rightarrow \ell^{2}([a,b])$. $\CL_{[a,b]}, \CM_{[a,b]}$ are defined similarly.

With $\beta, \gamma \in \partial\D$, define the sequence of Verblunsky coefficients
$$
\tilde{\alpha}_{n} =
   \begin{cases}
   \beta,  \quad & n=a;  \\
   \gamma, \quad &n=b; \\
   \alpha_{n}, \quad & n \notin \{a,b\}. \\
   \end{cases}
$$
The corresponding operator is denoted by $\tilde{\CE}_{\omega}$ and we define $\CE_{\omega,[a,b]}^{\beta,\gamma}$ by $P_{[a,b]}\tilde{\CE_{\omega}}(P_{[a,b]})^{*}$. One can verify that $\CE_{\omega,[a,b]}^{\beta,\gamma}$ is unitary whenever $\beta,\gamma \in \pD$.

Then, for $z \in \mathbb{C}, \beta, \gamma \in \pD$, we define the polynomials
\[
\Phi_{\omega,[a,b]}^{\beta,\gamma}(z)
=
\det\left[z - \CE_{\omega,[a,b]}^{\beta,\gamma}\right], \quad
\varphi_{\omega,[a,b]}^{\beta,\gamma}(z)=(\rho_{a}\cdots\rho_{b})^{-1}\Phi_{\omega,[a,b]}^{\beta,\gamma}(z).
\]

\subsection{Estimating Transfer Matrices}

Let us recall \cite[Corollary~3.11, Lemma 3.12, and Lemma~4.4]{Krueg2013IMRN}, which give the connection between $\varphi_{\omega,[a,b]}^{\beta,\gamma}(z)$ and $\|S^z_{b-a+1}(\omega)\|$. An immediate corollary of them is the following:

\begin{coro}\label{coroupperbound}
For $\beta, \gamma \in \pD$, we have the following estimates,
    \[
\left| \varphi_{\omega,[a,j-1]}^{\beta,\alpha_{j-1}} (z) \right|
\leq
\sqrt{2} \|S^z_{j-a}(T^a \omega)\|,
\quad
\left| \varphi_{\omega,[k+1,b]}^{\alpha_{k+1},\gamma}(z) \right|
\leq
\sqrt{2}\|S^z_{b-k}(T^{k+1}\omega)\|.
\]
Moreover, for $\beta_{0}, \gamma_{0} \in \pD$, there exist
$$
\beta \in \{\beta_{0}, -\beta_{0}\}, \quad \gamma \in \{\gamma_{0}, -\gamma_{0}\}
$$
such that
\begin{equation*}
|\varphi_{\omega, [a,b]}^{ \beta,  \gamma}|>C\|S_{b-a+1}^{z}(T^{a}\omega)\|.
\end{equation*}
\end{coro}

\begin{proof}
The first two inequalities can be easily obtained from \cite[Corollary~3.11 and Lemma 3.12]{Krueg2013IMRN}. From the proof of \cite[Theorem 4.4]{Krueg2013IMRN}, we may find that
\begin{equation*}
    \begin{pmatrix} \varphi_{\omega, [a,b]}^{\beta, \gamma} & \varphi_{\omega, [a,b]}^{-\beta, \gamma} \\ \varphi_{\omega, [a,b]}^{\beta, -\gamma} & \varphi_{\omega, [a,b]}^{-\beta, -\gamma} \end{pmatrix}
    =
    \begin{pmatrix} z & -\bar{\gamma} \\ z & \bar{\gamma} \end{pmatrix}
    S^{z}_{b-a+1}(T^{a}\omega)
    \begin{pmatrix} 1 & 1 \\ \beta & -\beta \end{pmatrix},
\end{equation*}
which means that $\max\{ |\varphi_{\omega, [a,b]}^{\pm \beta, \pm \gamma}| \}>C\|S_{b-a+1}^{z}(T^{a}\omega)\|$.
\end{proof}

\subsection{Green's Function Estimates}{\label{aegf}}

Since $\CL_{\omega}$ is unitary, we know that $\CE_\omega \psi = z\psi$ is equivalent to $(z\CL_\omega^* - \CM_\omega)\psi = 0$. Then the associated finite-volume Green's functions are defined by
\[
G_{\omega,[a,b]}^{\beta,\gamma}(z)
=
\left(z\left[\CL_{\omega,[a,b]}^{\beta,\gamma}\right]^* - \CM_{\omega,[a,b]}^{\beta,\gamma} \right)^{-1}
\]
and
\[
G_{\omega,[a,b]}^{\beta,\gamma}(j,k;z)
=
\langle \delta_j, G_{\omega,[a,b]}^{\beta,\gamma}(z)
 \delta_k \rangle,
\quad
j,k \in [a,b].
\]
By \cite[Proposition~3.8]{Krueg2013IMRN}, for $\beta, \gamma \in \pD$, the Green's function has the following expression
\begin{equation} \label{CMVGramma}
\left| G_{\omega,[a,b]}^{\beta,\gamma}(j,k;z) \right|
=
\frac{1}{\rho_j\rho_k}
\left| \frac{\varphi_{\omega,[a,j-1]}^{\beta,\alpha_{j-1}} (z) \varphi_{\omega,[k+1,b]}^{\alpha_{k+1},\gamma}(z) }{\vspace{.5cm}\varphi_{\omega,[a,b]}^{\beta,\gamma}(z)} \right|,
\quad
 a \leq j \leq k \leq b.
\end{equation}

Now we can get the following Green's function estimate.

\begin{lemma} \label{lemma:LEandGreenEst:CMV}
Assume that for $n$ large enough, the following inequality holds true,
\begin{equation*} \label{eq:normupbound}
\frac{1}{n} \log\|M_n^z(\omega) \|
\geq
L_{n}(z) - \e.
\end{equation*}
Then we have, for any $\beta_{0},\gamma_{0}\in \pD$, there exist $\beta \in \{\beta_{0}, -\beta_{0}\}$, $\gamma \in \{\gamma_0, -\gamma_0\}$, such that
\begin{equation} \label{eq:greenfnest1:CMV}
\left| G^{\beta,\gamma}_{\omega,[0,n)}(j,k;z) \right|
\leq e^{-|j-k|L_{n}(z) + C \e n}
\end{equation}
for all $j,k \in [0,n)$, $z \in \pD \setminus \sigma(\CE_{\omega , [0,n)}^{\beta,\gamma})$.
\end{lemma}

\begin{proof}
Noticing that $\|S^{z}_{n}(\omega)\|=\|M^{z}_{n}(\omega)\|$, Lemma~\ref{lub} and the first two inequalities of Corollary~\ref{coroupperbound} allow us to bound the numerator in equation \eqref{CMVGramma} from above, while the assumption and the last inequality of Corollary~\ref{coroupperbound} allow us to estimate the denominator from below. Combining these observations with the boundedness of $\frac{1}{\rho_{j}\rho_{k}}$, the desired conclusion follows.
\end{proof}

\section{Elimination of Double Resonances and Semi-Algebraic Sets}

In this section we prove a statement that is usually referred to as the {\it elimination of double resonances}. From the proof of \cite{BG00}, one may see that the uniform positivity of and a uniform LDT for the Lyapunov exponent, together with the elimination of double resonances, imply Anderson localization.

\subsection{Elimination of Double Resonances at One Point}

For the restriction $H_{[0,n-1]}$ of a Schr\"odinger operator $H$, since it is self-adjoint, we have
\begin{equation*}
   \text{ dist$(E, \sigma(H_{[0,n-1]})) = \| (E - H_{[0,n-1]})^{-1} \|^{-1}$ for $E \notin \sigma(H_{[0,n-1]})$, }
\end{equation*}
which is necessary for the elimination of double resonances in Schr\"odinger case. Actually, this formula holds true for normal operators, and in particular for the unitary operators $\CE^{\beta, \gamma}_{\omega, [a, b]}$ (this matrix is unitary whenever $\beta, \gamma \in \pD$ by Section~\ref{section3}).

Consequently, we can get the following CMV analog of \cite[Lemma 4.1]{BG00}, where the constant $C(\alpha)$ appearing in it is defined by \begin{equation}\label{e.cofalphadef}
C(\alpha) = \sup_{n \in \Z_+} \left( n e^{n C_\alpha} \right)^{1/n}
\end{equation}
with $C_\alpha$ from \eqref{e.calphadef}.

\begin{lemma}{\label{lee}}
    Fix $x_0 \in \T$, $\beta, \gamma \in \pD$, and $\kappa > 0$, and assume that $\log \log \bar{n} \ll \log n$ {\rm (}so that $\bar{n}^{2}e^{-(c_{1}/C_{\alpha}^{3})\kappa ^{3}n} < e^{-\frac{1}{2}(c_{1}/C_{\alpha}^{3})\kappa ^{3}n}${\rm )}. Denote by $S \subset \mathbb{T} \times \mathbb{T}$ the set of $(\omega, x)$ such that
    \begin{equation}{\label{dcac}}
      \omega \in \mathrm{DC};
    \end{equation}
    \begin{equation}{\label{cee}}
        \text{there are $n_{0}< \bar{n}$ and $z \in \partial \mathbb{D}$ such that}
    \end{equation}
    \begin{equation}{\label{iee1}}
        \lVert (z-\CE^{\beta, \gamma}_{\omega, [0, n_{0}-1]}(x_0))^{-1} \rVert > C(\alpha)^{n}
    \end{equation}
and
    \begin{equation}{\label{iee2}}
        \frac{1}{n}\log \lVert M^{z}_{n}(x) \rVert <L_{n}(z) - \kappa.
    \end{equation}
Then
    \begin{equation}{\label{mesdr}}
       \text{ {\rm mes} $S < e^{-\frac{1}{16}(c_{1}/C_{\alpha}^{3})\kappa ^{3}n}$}.
    \end{equation}
\end{lemma}

\begin{proof}
For an $\omega \in \mathrm{DC}$ and $\beta, \gamma \in \pD$ as fixed above, denote
\begin{equation*}
    \Lambda_{\omega}^{\beta, \gamma} = \bigcup_{n_{0}<\bar{n}} \sigma(\mathcal{E}^{\beta, \gamma}_{\omega, [0, n_{0}-1]}(x_{0})))
\end{equation*}
and
\begin{equation*}
    \Theta_{\omega}^{\beta, \gamma} = \Big \{ x \in \mathbb{T} : \max_{z_{1}\in \Lambda^{\beta, \gamma}_{\omega}} \Big | \frac{1}{n} \log \| M^{z_{1}}_{n}(x) \|  - L_{n}(z_{1})\Big | > \frac{\kappa}{2} \Big  \}.
\end{equation*}
Then, by Lemma~\ref{lldt}, we have
\begin{equation*}
    \text{{\rm mes}  $\Theta_{\omega}^{\beta, \gamma} < \bar{n}^{2}e^{- \frac18 (c_{1}/C_{\alpha}^{3})\kappa ^{3}n} < e^{-\frac{1}{16}(c_{1}/C_{\alpha}^{3})\kappa ^{3}n}$}.
\end{equation*}
Hence \eqref{mesdr} will follow from the fact that
\begin{equation}{\label{ssub}}
    S \subset \{ (\omega, x) \in \mathbb{T} \times \mathbb{T} : \omega \in \mathrm{DC}, \; x \in \Theta_{\omega}^{\beta, \gamma} \}.
\end{equation}

Assume \eqref{iee1} and \eqref{iee2} hold for some $z \in \partial\mathbb{D}$. Since $\CE^{\beta, \gamma}_{\omega, [0, n_{0}-1]}(x_{0})$ is unitary,
\begin{equation*}
    {\rm dist}(z, \sigma(\mathcal{E}^{\beta, \gamma}_{\omega, [0, n_{0}-1]}(x_{0}))) =   \| (z-\mathcal{E}^{\beta, \gamma}_{\omega, [0, n_{0}-1]}(x_{0}))^{-1} \|^{-1},
\end{equation*}
which means that there exists $z_{1} \in \sigma(\mathcal{E}^{\beta, \gamma}_{\omega, [0, n_{0}-1]}(x_{0}))$ such that $|z-z_{1}| < C(\alpha)^{-n}$, and hence, by \eqref{e.cofalphadef},

$$
    \| M_{n}^{z}(x) - M^{z_{1}}_{n}(x) \| \leq ne^{nC_\alpha} |z-z_{1}| < ne^{nC_\alpha} C(\alpha)^{-n} \le 1
$$
for every $x \in \T$.

It follows from the fact $|\log a - \log b| \leq |a - b|$ ($a, b \geq 1$) that
\begin{equation*}
   \Big |\log \|M^{z}_{n}(x)\| - \log \| M^{z_{1}}_{n}(x) \| \Big| < 1
\end{equation*}
and
\begin{equation*}
    |L_{n}(z) - L_{n}(z_{1})| < \frac{1}{n}.
\end{equation*}
Therefore, for $n$ large enough,
\begin{equation*}
    \begin{split}
        \frac{1}{n} \log \| M^{z_{1}}_{n}(x) \|
        &< \frac{1}{n} \log \| M^{z}_{n}(x) \| + \frac{1}{n} \\
        &< L_{n}(z) - \kappa + \frac{\kappa}{4} \\
        &< L_{n}(z_{1}) - \frac{1}{2} \kappa,
    \end{split}
\end{equation*}
and hence $x \in \Theta_{\omega}^{\beta, \gamma}$, which implies \eqref{ssub}.
\end{proof}

\subsection{Semi-Algebraic Sets}

In order to get the elimination of double resonances along the orbit $\{x+l\omega\}$, we need complexity bounds for semi-algebraic sets.

We denote the Hilbert-Schmidt norm of a matrix $B$ as
\begin{equation*}
    \| B \|_{HS} = \Big (\sum_{i, j} |B_{ij}|^{2} \Big )^{1/2}.
\end{equation*}

In \cite[Section~5]{BG00}, Bourgain and Goldstein, considering Schr\"odinger operators with trigonometric polynomial potentials, reformulate the two resonance conditions (\cite[(4.4) and (4.5)]{BG00}) as polynomial inequalities in $(\cos\omega, \sin\omega, \cos\theta, \sin\theta, E)$ and get an upper bound for the number of connected components of the exceptional set $S$.

In the case of general real analytic functions, one can replace them by the truncation of their Fourier expansion of degree $< n^{2}$, with error $< e^{-cn^{2}}$ (this is because the Fourier coefficients of an analytic function decay exponentially), which is sufficient for the measure estimate of $S$.

In our CMV case, from the proof of Lemma~\ref{lldt}, we know that the condition \eqref{iee2} is equivalent to the following one,
\begin{equation*}
     \frac{1}{n}\log \lVert A_{n}(\omega, x_{0}+l\omega, z) \rVert < \tilde{L}_{n}(\omega, z) - \kappa,
\end{equation*}
which can be dealt with as the real analytic case.

For the condition \eqref{iee1}, its left-hand side is a polynomial in $(\Re \, {\alpha(x)}, \Im \, {\alpha(x)}, \Re \, {z}, \Im \, {z})$, which also can be dealt with by the above method since $\Re \, {\alpha(x)}$ and $\Im \, {\alpha(x)}$ are real analytic.

More precisely, condition \eqref{iee1} can be replaced by
\begin{equation*}
    \|(z-\mathcal{E}^{\beta, \gamma}_{\omega,[0, n_{0}-1]}(x_{0}))^{-1}\|_{HS}^{2} > C^{2n},
\end{equation*}
and hence
\begin{equation*}
    \begin{split}
        \sum_{1\leq n_{1}, n_{2}\leq n_{0}} & |\det[(n_{1}, n_{2}) - \text{minor of } (z - \mathcal{E}^{\beta, \gamma}_{\omega,[0, n_{0}-1]}(x_{0}))]|^{2} \\
        &> C^{2n}|\det(z - \mathcal{E}^{\beta, \gamma}_{\omega,[0, n_{0}-1]}(x_{0}))|^{2},
    \end{split}
\end{equation*}
which is of the form
\begin{equation}{\label{iee1r}}
    P_{1}(\cos\omega, \sin\omega, \Re \, {z}, \Im \, {z}) > 0,
\end{equation}
where $P_{1}(\cos\omega, \sin\omega, \Re \, {z}, \Im \, {z})$ is a polynomial of degree at most $Cn_{0}^{2}n^{2}$ in $\cos\omega, \sin\omega$ and at most $Cn_{0}$ in $\Re \, {z}, \Im \, {z}$.

Considering condition \eqref{iee2}, we can replace $L_{n}(\omega, z)$ by using Lemma \ref{lar}. Taking $J = [n^{2A}]$, it follows from \eqref{ear} that condition \eqref{iee2} is equivalent to
\begin{equation*}
    \frac{1}{n}\log \lVert M_{n}(\omega, x , z) \rVert < \frac{1}{nJ}\sum_{1}^{J} \log \|M_{n}(\omega, x_{0}+j\omega, z)\| - C\kappa,
\end{equation*}
and hence
\begin{equation*}
    \| M_{n}(\omega, x , z) \|_{HS}^{2J} < C^{J} e^{-C n \kappa J} \prod_{1}^{J} \| M_{n}(\omega, x_{0}+j\omega, z) \|^{2}_{HS}.
\end{equation*}
This condition is of the form
\begin{equation}{\label{iee2r}}
    P_{2}(\cos\omega, \sin\omega, \cos x, \sin x, \Re \, {z}, \Im \, {z}) > 0,
\end{equation}
where $P_{2}(\cos\omega, \sin\omega, \cos x, \sin x, \Re \, {z}, \Im \, {z})$ is a polynomial of degree at most $n^{C_{A}}$.

Similarly to \cite[Lemma 5.13]{BG00}, we have

\begin{lemma}\label{lsemi}
In Lemma~\ref{lee}, take $\bar{n} < n^{C}$ and replace conditions \eqref{iee1}, \eqref{iee2} by \eqref{iee1r}, \eqref{iee2r}. Then, for fixed $x_{0}$, the set of $\omega$ satisfying \eqref{dcac} and \eqref{cee} is a union of at most $n^{C}$ intervals.
\end{lemma}

Recall the following general frequency estimate \cite[Lemma 6.1]{BG00}:

\begin{prop}{\label{pfe}}
Let $S \subset \mathbb{T} \times \mathbb{T}$ be a set with the following property:
\begin{equation*}
\text{For each $x \in \mathbb{T}$, the section $S_{x}=\{ \omega \in \mathbb{T} : (\omega, x) \in S \}$ is a union of at most M intervals.}
\end{equation*}
Then, for $N$ sufficiently large, we have
\begin{equation*}
{\rm mes} \left\{ \omega \in [0,1] : (\omega, x_{0}+l\omega) \in S \text{ for some } N \le l \le 2N-1 \right\} \leq N^{3}({\rm mes} \, S)^{1/2} + MN^{-1}.
\end{equation*}
\end{prop}

Combining this proposition with Lemma~\ref{lee} and Lemma~\ref{lsemi} above, similarly to \cite[Lemma 6.17]{BG00}, we can get the following estimate.

\begin{lemma}{\label{ll}}
     Fix $\kappa > 0$, $\beta, \gamma \in \pD$, and a sufficiently large integer $n$. Denote by $\Omega_{n,\kappa} \subset \T$ the set of frequencies $\omega \in \mathrm{DC}$ satisfying the following:
    \begin{equation}{\label{cll}}
        \text{There are $n_{0} < n^{C}$, $2^{(\log n)^{2}} \leq l \leq 2^{(\log n)^{3}}$, and $z \in \partial \mathbb{D}$ such that}
    \end{equation}
    \begin{equation}{\label{ill1}}
        \lVert (z-\CE^{\beta, \gamma}_{\omega, [0, n_{0}-1]}(x_0))^{-1} \rVert > C(\alpha)^{n},
    \end{equation}
    \begin{equation}{\label{ill2}}
        \frac{1}{n}\log \lVert M_{n}(\omega, x_{0}+l\omega, z) \rVert <L_{n}(\omega, z) - \kappa.
    \end{equation}
Then
    \begin{equation}{\label{lill}}
       \text{ {\rm mes} $\Omega_{n,\kappa} < 2^{-\frac{1}{4}(\log n)^{2}}$}.
    \end{equation}
\end{lemma}

\begin{proof}
    Let $2^{(\log n)^{2}} \leq N \leq 2^{(\log n)^{3}}$ and $\bar{n} = n^{C}$ in Lemma \ref{lee}. The set $S \subset \mathbb{T} \times \mathbb{T}$ of $(\omega, x)$ satisfying $\omega \in \mathrm{DC}$, \eqref{iee1r} and \eqref{iee2r} is of measure
\begin{equation*}
    \text{ {\rm mes} $S < e^{-\frac{1}{16}(c_{1}/C_{\alpha}^{3})\kappa ^{3}n}$.}
\end{equation*}
Moreover, by Lemma~\ref{lsemi}, each of the sections $S_{x}$ of $S$ is a union of at most $n^{C}$ intervals.

Hence Proposition \ref{pfe} implies
\begin{align*}
\text{mes} \{ \omega \in \mathrm{DC}  & : \text{ we have \eqref{ill1}--\eqref{ill2} for some } n_{0} < n^{C}, \, N \le l \le 2N-1 \text{ and } z \in \partial \mathbb{D} \} \\
& \leq \text{mes} \{ \omega \in \mathbb{T} : (\omega, x_{0}+l\omega) \in S \text{ for some } N \le l \le 2N-1 \} \\
& \lesssim N^{3} e^{-\frac{1}{32}(c_{1}/C_{\alpha}^{3})\kappa ^{3}n} + n^{C}N^{-1} \\
& \leq N^{-\frac{1}{2}}.
\end{align*}
Summing over the ranges of $N$ implies \eqref{lill}.
\end{proof}


\section{Proof of Theorem~\ref{talcmv}}

In fact, if we replace \eqref{ill1} by the condition
\begin{equation*}
    \lVert (z-\mathcal{E}^{\beta, \gamma}_{\omega,[-n_{0}, n_{0}]}(x_0))^{-1} \rVert > C(\alpha)^{n}
\end{equation*}
and restrict the index set to $[-n_{0}, n_{0}]$ rather than to $[0, n_{0}-1]$, we can also get the estimate \eqref{lill}.

Denote by $\Omega_{n,\kappa}$ the frequency set obtained in Lemma~\ref{ll} and define
\begin{equation*}
\Omega_{\kappa} = \bigcap\limits_{n'} \bigcup\limits_{n>n'} \Omega_{n,\kappa}, \quad \Omega = \bigcup\limits_{\kappa} \Omega_{\kappa}.
\end{equation*}
Then we have
\begin{equation*}
    \text{mes $\Omega_{\kappa} \le \inf\limits_{n'} \sum\limits_{n>n'} e^{-\frac{1}{4}(\log n)^{2}} =0$}.
\end{equation*}
Since $\Omega_{\kappa}$ is decreasing, we can take a countable subsequence of $\kappa$ and hence obtain $\text{mes} \, \Omega = 0$.

Assume $\omega \in \mathcal{I} \cap ({\rm DC} \backslash \Omega)$ and let $z \in \mathcal{K}$, $\xi = (\xi_{n})_{n \in \mathbb{Z}}$ satisfy the equation
\begin{equation}{\label{eef}}
    \mathcal{E}_{\omega} \xi = z \xi,
\end{equation}
where
\begin{equation}
    \xi _{0} = 1  \quad {\rm and} \quad  |\xi_{n}| \lesssim |n|^{C}.
\end{equation}
Assume $L(z) > \delta_{0} > 0$ and $\kappa \ll \delta_{0}$. Since $\omega \notin \Omega_{\kappa}$, there exists $n'$ such that $\omega \notin \Omega_{n, \kappa}$ for all $n>n'$.

Since $L_{n}(z) \to L(z)$, we can take $n'$ large enough to satisfy
\begin{equation*}
    \text{$L_{n}(z) < L(z) + \kappa$ for $n > n'$}.
\end{equation*}

Since condition \eqref{cll} is not fulfilled, then if there is $n_{0} < n^{C}$ such that
\begin{equation*}
    \lVert G^{\beta, \gamma}_{\omega, [-n_{0}, n_{0}]}(x_0) \rVert > C(\alpha)^{n},
\end{equation*}
Lemma~\ref{ll} implies that for all $2^{(\log n)^{2}} \leq \lvert l \rvert  \leq 2^{(\log n)^{3}}$, \eqref{ill2} must fail, that is,
\begin{equation}
    \frac{1}{n} \log \lVert M^{z}_{n}(x_{0} + l\omega) \rVert > L_{n}(z) - \kappa.
\end{equation}
Then, by Lemma~\ref{lemma:LEandGreenEst:CMV}, there exist $\beta$ and $\gamma$ such that for each $2^{(\log n)^{2}} \leq l \leq 2^{(\log n)^{3}}$,
\begin{equation}{\label{egI}}
    \lvert G^{\beta, \gamma}_{\omega, [l,n-1+l]}(n_{1}, n_{2}) \rvert \leq e^{-L(z)\lvert n_{1}-n_{2} \rvert + C\kappa n}<e^{-\delta_{0}\lvert n_{1} - n_{2} \rvert + C\kappa n}.
\end{equation}
Now consider the interval $[\frac{N}{2}, 2N]$, where $2^{(\log n)^{2}+2} \leq N \leq 2^{(\log n)^{3}-1}$ (this makes sure that $[\frac{N}{2}, 2N] \subset [k+2^{(\log n)^{2}}, n-1+2^{(\log n)^{3}}]$).

Invoking the paving property (Appendix~\ref{paving}), we can deduce from \eqref{egI} that
\begin{equation}
    \text{$\lvert G^{\beta, \gamma}_{\omega, [\frac{N}{2}, 2N]}(n_{1}, n_{2}) \rvert < e^{-\frac{\delta_{0}}{2} \lvert n_{1}-n_{2} \rvert + C\kappa N}$ for $n_{1}, n_{2} \in [\frac{N}{2}, 2N]$.}
\end{equation}

Restricting the equation \eqref{eef} to $[\frac{N}{2}, 2N]$, for $k \in [\frac{N}{2}, 2N]$, we have (for details, see Appendix~\ref{boundary}),
\begin{equation}{\label{iefg}}
    \lvert \xi_{k} \rvert \leq |G^{\beta, \gamma}_{\omega, [\frac{N}{2}, 2N]}(k, \frac{N}{2};z)| N^{C} + |G^{\beta, \gamma}_{\omega, [\frac{N}{2}, 2N]}(k, 2N;z)| N^{C},
\end{equation}
and hence
\begin{equation}
    \lvert \xi_{N} \rvert \leq N^{C} e^{-\frac{\delta_{0}}{4}N + C\kappa N}< e^{-\frac{\delta_{0}}{5}N}.
\end{equation}
This is the required exponential decay property (also valid for the negative side).

It remains to show that for some $n_{0} < n^{C}$, the following inequality holds,
$$
\|G_{\omega,[-n_{0}, n_{0}]}^{\beta, \gamma}\| > C(\alpha)^{n}.
$$

\noindent Recalling $\xi_{0}=1$, \cite[Lemma 3.9]{Krueg2013IMRN} implies
\begin{equation}
    1 \leq \lVert \xi_{[-n_{0},n_{0}]} \rVert \leq 2 \lVert  G^{\beta, \gamma}_{\omega, [-n_{0},n_{0}]} \rVert (\lvert \xi_{-n_{0}} \rvert +\lvert \xi_{-n_{0}+1} \rvert  + \lvert \xi_{n_{0}-1} \rvert +\lvert \xi_{n_{0}} \rvert).
\end{equation}
Thus it will suffice to show that there exists $n_{0} < n^{C}$ such that
\begin{equation}{\label{xin0}}
    \lvert \xi_{-n_{0}} \rvert +\lvert \xi_{-n_{0}+1} \rvert  + \lvert \xi_{n_{0}-1} \rvert +\lvert \xi_{n_{0}} \rvert < \frac{1}{2}C(\alpha)^{-n}.
\end{equation}
Let
\begin{equation*}
    n_{1} = C' \frac{n}{\delta_{0}}
\end{equation*}
where $C'=\frac9n \log2+3 \log C(\alpha)$.

Assume for some $0 < j < n^{C}$,
\begin{equation}{\label{n11}}
   \Big | \frac{1}{n_{1}} \log \| M^{z}_{n_{1}} (x_{0} + j \omega) \| - L_{n_{1}}(z) \Big | < C\kappa.
\end{equation}
Then we have
\begin{equation*}
    |G^{\beta, \gamma}_{\omega, [j, n_{1}-1 + j]}(x, y)| < e^{-L(z)|x-y| + C\kappa n_{1}},
\end{equation*}
which implies for $k = j + \LARGE [ \frac{n_{1}}{2}\LARGE ]$,
\begin{equation*}
    \begin{split}
        |\xi_{k}| &<  n^{C} |G^{\beta, \gamma}_{\omega, [ j, n_{1}-1 + j ]}(k, j)| + |G^{\beta, \gamma}_{\omega, [ j, n_{1}-1 + j ]}(k, n_{1}-1+j)|\\
        &< n^{C} e^{-L(z)\frac{n_{1}}{2} + C\kappa n_{1}}\\
        &< e^{-\frac{\delta_{0}}{3}n_{1}} = \frac{1}{8}C(\alpha)^{-n}.
    \end{split}
\end{equation*}
Obviously, we also have $|\xi_{k-1}| < \frac{1}{8}C(\alpha)^{-n}$.

Similarly, assuming
\begin{equation}{\label{n12}}
     \Big |\frac{1}{n_{1}} \log \| M^{z}_{n_{1}} (x_{0} + (-j-n_{1}) \omega) \| - L_{n_{1}}(z) \Big | < C\kappa,
\end{equation}
the same method yields $ |\xi_{-k}|, |\xi_{-k+1}| < \frac{1}{8}C(\alpha)^{-n}$.

Since $k < n^{C}$ implies that \eqref{xin0} is satisfied, it remains to show that \eqref{n11} and \eqref{n12} hold.

Letting $J=n^{C}$, we verify them by averaging over $j \in [J, 2J]$. Thus recalling Lemma~\ref{lar} (with $n$ replaced by $n_{1}$), we see that
\begin{equation}
    \begin{split}
    \frac{1}{J} \sum_{j=J+1}^{2J} & \Big [ \frac{1}{n_{1}} \log \lVert M^{z}_{n_{1}}(x_{0}+j \omega)\rVert + \frac{1}{n_{1}} \log \lVert M^{z}_{n_{1}}(x_{0}+(-j-n_{1}) \omega)\rVert \Big ]\\
    &= 2L_{n_{1}}(z) + O \Big ( \frac{1}{n_{1}} \Big).
    \end{split}
\end{equation}
Hence, there is $J < j \leq 2J$ such that
\begin{equation}
    \begin{split}
    \frac{1}{n_{1}} \log & \lVert M^{z}_{n_{1}}(x_{0}+j \omega)\rVert + \frac{1}{n_{1}} \log \lVert M^{z}_{n_{1}}(x_{0}+(-j-n_{1}) \omega)\rVert\\
    & > 2L_{n_{1}}(z) + O \Big ( \frac{1}{n_{1}} \Big),
    \end{split}
\end{equation}
implying by the upper bound (Lemma~\ref{lub})
\begin{equation}
    \begin{split}
    L_{n_{1}}(z) + C\kappa
    & >\frac{1}{n_{1}} \log  \lVert M^{z}_{n_{1}}(x_{0}+j \omega)\rVert\\
    & > 2L_{n_{1}}(z) + O \Big ( \frac{1}{n_{1}} \Big) - (L_{n_{1}}(z) + C\kappa)\\
    & > L_{n_{1}}(z) - C\kappa.
    \end{split}
\end{equation}
Hence \eqref{n11} holds, and in the same way we may obtain the estimate \eqref{n12}.

Therefore $\mathcal{E}_{\omega}(x_{0})$ exhibits Anderson localization for a.e.\ $\omega \in \mathcal{I}$, assuming $L(z) > \delta_{0}$.

\begin{remark}
Recall that $\mathcal{C}_{[a,b]} = \CE_{[a,b]}$ whenever $1 \leq a \leq b$ and $\mathcal{C}_{[0,b]}=\CE_{[0,b]}$ with $\alpha_{-1} = -1$. Using this, the proof of the half-line case, Theorem \ref{halflineal}, can be carried out similarly.
\end{remark}

\section{Quantum Walks}

In this section, we apply our general localization result to quantum walks and obtain Anderson localization for quantum walks with suitable analytic quasi-periodic coins.

A \emph{quantum walk} is described by a unitary operator on the Hilbert space $\mathcal{H} = \ell^{2}(\mathbb{Z}) \otimes \mathbb{C}^{2}$, which models a state space in which a wave packet comes equipped with a spin at each integer site. Here, the elementary tensors of the form $\delta_{n} \otimes e_{\uparrow}$ and $\delta_{n} \otimes e_{\downarrow}$ with $n \in \mathbb{Z}$ comprise an orthonormal basis of $\mathcal{H}$ (where $\{ e_{\uparrow}, e_{\downarrow} \}$ denotes the canonical basis of $\mathbb{C}^{2}$). A time-homogeneous quantum walk scenario is given as soon as unitary coins
\begin{equation*}
    C_{n} = \begin{pmatrix}
              c^{11}_{n}     &    c^{12}_{n} \\
              c^{21}_{n}     &    c^{22}_{n}
    \end{pmatrix} \in \mathbb{U}(2), \quad n \in \mathbb{Z},
\end{equation*}
are specified. As one passes from time $t$ to time $t + 1$, the update rule of the quantum walk applies the coins coordinate-wise and shifts spin-up states to the right and spin-down states to the left, viz
\begin{equation*}
    \delta_{n} \otimes e_{\uparrow} \mapsto c^{11}_{n}\delta_{n+1} \otimes e_{\uparrow} + c^{21}_{n}\delta_{n-1} \otimes e_{\downarrow},
\end{equation*}
\begin{equation*}
    \delta_{n} \otimes e_{\downarrow} \mapsto c^{12}_{n}\delta_{n+1} \otimes e_{\uparrow} + c^{22}_{n}\delta_{n-1} \otimes e_{\downarrow}.
\end{equation*}
If we extend this by linearity and continuity to general elements of $\mathcal{H}$, this defines a unitary operator $U$ on $\mathcal{H}$. Equivalently, denote a typical element $\psi \in \mathcal{H}$ by
\begin{equation*}
    \psi = \sum_{n \in \mathbb{Z}}(\psi_{\uparrow, n}\delta_{n} \otimes e_{\uparrow} + \psi_{\downarrow, n}\delta_{n} \otimes e_{\downarrow}),
\end{equation*}
where one must have
\begin{equation*}
    \sum_{n \in \mathbb{Z}}(\lvert \psi_{\uparrow, n} \rvert ^{2} + \lvert \psi_{\downarrow, n} \rvert ^{2}) < \infty.
\end{equation*}
We may then describe the action of $U$ in coordinates via
\begin{equation*}
    [U \psi]_{\uparrow, n} = c_{n-1}^{11} \psi_{\uparrow, n-1} + c_{n-1}^{12} \psi_{\downarrow, n-1},
\end{equation*}
\begin{equation*}
    [U \psi]_{\downarrow, n} = c_{n+1}^{21} \psi_{\uparrow, n+1} + c_{n+1}^{22} \psi_{\downarrow, n+1}
\end{equation*}
and the matrix representation of a quantum walk is given by
\begin{equation*}
    U =
    \begin{pmatrix}
         \ddots & \ddots & \ddots &     &   &    &   &     \\
              & 0 & 0  &  c^{21}_{0}  &  c^{22}_{0} &  &   &   &    \\
              &  c^{11}_{-1} &  c^{12}_{-1}  &0  &  0 &  &   &    &      \\
               & &   & 0 & 0  &  c^{21}_{1}  &  c^{22}_{1} &       &           \\
               & &   &c^{11}_{0} &  c^{12}_{0}  &0  &  0 &       &               \\
                & &   &   &    & 0 & 0  &  c^{21}_{2}  &  c^{22}_{2}  &         \\
                & &   &   &    & c^{11}_{1} &  c^{12}_{1} & 0 &0  &         \\
                & &   &   &    &    &  \ddots & \ddots & \ddots
    \end{pmatrix}.
\end{equation*}

Quantum walks can be connected to extended CMV matrices as follows. If all Verblunsky coefficients with even index vanish, then the extended CMV matrix becomes
\begin{equation*}
    \mathcal{E} =
    \begin{pmatrix}
         \ddots & \ddots & \ddots &     &   &    &   &     \\
              & 0 & 0  &  \overline{\alpha_{1}}  &  \rho_{1} &  &   &   &    \\
              &  \rho_{-1} &  -\alpha_{-1}  &0  &  0 &  &   &    &      \\
               & &   & 0 & 0  &  \overline{\alpha_{3}} &  \rho_{3} & &   \\
               & &   &\rho_{1} &  -\alpha_{1}  &0  &  0 &       &               \\
                & &  &  &  & 0 & 0  & \overline{\alpha_{5}} &  \rho_{5}  &    \\
                & &   &   &    & \rho_{3} &  -\alpha_{3} & 0 &0  &         \\
                & &   &   &    &    &  \ddots & \ddots & \ddots
    \end{pmatrix},
\end{equation*}
which resembles the matrix representation $U$.

One may notice, however, that $\rho_{n}>0$ in the extended CMV matrix $\mathcal{E}$, whereas $c_{n}^{kk}$ in the quantum walk matrix $U$ could be a complex number with non-zero argument, which means that these two matrices may not match. However, this can be easily resolved by conjugation with a suitable diagonal unitary, that is, there exists a diagonal unitary matrix $\Lambda$ such that $\mathcal{E} = \Lambda^{*} U \Lambda$. For details, see \cite{CGMV10} or \cite{DFO16}.

Therefore, the discussion above shows that a quantum walk corresponds to an extended CMV matrix whose Verblunsky coefficients with even index are zero.

From now on we consider the case of coined quantum walks on the integer lattice where the coins are distributed quasi-periodically according to the rule

\begin{equation*}
    C_{n} = C_{n, \omega, x}= \begin{pmatrix}
              c^{11}(x+n\omega)     &    c^{12}(x+n\omega) \\
              c^{21}(x+n\omega)     &    c^{22}(x+n\omega)
    \end{pmatrix}
\end{equation*}
with analytic sampling functions and $\omega \notin \mathbb{Q}$. We denote the corresponding quantum walk by $U_{\omega}(x)$. In this section, in order to distinguish the energy and the complex number $z=x+iy$, we will use $E$ to denote the energy.

\subsection{Gesztesy-Zinchenko Cocycles}{\label{GZcocycle}}

It is well known that for a Schr\"odinger operator $H_{\omega}$, the solutions of the difference equation $H_{\omega}u = Eu$ can be generated by transfer matrices, which for dynamically generated potentials in turn are generated by a suitable Schr\"odinger cocycle. In \cite{GZ06}, Gesztesy and Zinchenko construct an analogous cocycle (known as the G-Z cocycle) for the CMV case, which is basically a two step Szeg\H o cocycle.

We will use a modified G-Z cocycle here based on notations in \cite{FOZ17}. Consider the matrices
\begin{equation*}
    M_{f}^{E}(x)=\frac{1}{\rho_{f}}\begin{pmatrix}
    -\alpha_{f}(x) & 1 \\ 1 & -\overline{\alpha_{f}(x)}
    \end{pmatrix}, \quad M_{g}^{E}(x)=\frac{1}{\rho_{g}}\begin{pmatrix}
    -\overline{\alpha_{g}(x)} & E \\ E^{-1} & -\alpha_{g}(x)
    \end{pmatrix}.
\end{equation*}
If $u_{x}=(u_{x}(n))_{n\in \mathbb{Z}} \in \mathbb{C}^{\mathbb{Z}}$ solves $\mathcal{E}_{\omega}(x)u_{x}=Eu_{x}$, then we define $v_{x}:=\mathcal{L}^{-1}u_{x}$. It follows that $\mathcal{M}u_{x}=Ev_{x}$ and we have $\mathcal{E}^{T}_{\omega}(x)v_{x}=Ev_{x}$. The G-Z cocycle recursion reads
\begin{equation*}
    \begin{pmatrix}
    u_{x}(k) \\ v_{x}(k)
    \end{pmatrix}=M_{g}^{E}(T^{\frac{k-1}{2}}x) \begin{pmatrix}
    u_{x}(k-1) \\ v_{x}(k-1)
    \end{pmatrix}, \quad \quad k\quad odd.
\end{equation*}
\begin{equation*}
    \begin{pmatrix}
    u_{x}(k) \\ v_{x}(k)
    \end{pmatrix}=M_{f}^{E}(T^{\frac{k-2}{2}}x) \begin{pmatrix}
    u_{x}(k-1) \\ v_{x}(k-1)
    \end{pmatrix},\quad \quad k \quad even.
\end{equation*}
The corresponding G-Z cocycle map $M^{E}:\T \rightarrow \mathbb{SU}(1, 1)$ is
\begin{equation*}
    M^{E}(x)=M^{E}_{f}(x)M^{E}_{g}(x)=\frac{1}{\rho_{f}(x)\rho_{g}(x)}\begin{pmatrix}
    \alpha_{f}(x)\overline{\alpha_{g}(x)} + E^{-1} & -\alpha_{g}(x)-E\alpha_{f}(x) \\
    -\overline{\alpha_{f}(x)}E^{-1} - \overline{\alpha_{g}(x)} & E+\overline{\alpha_{f}(x)}\alpha_{g}(x)
    \end{pmatrix}.
\end{equation*}

Define $
    P = \begin{pmatrix}
              0   &   1\\
              1   &   0
    \end{pmatrix}$.
Then the two step Szeg\H o cocycle map $S^{E}(x)$ can be related to $M^{E}(x)$ by
\begin{equation*}
    S^{E}(x) = P\cdot M^{E}(x)\cdot P.
\end{equation*}
This is the general setting of G-Z cocycle and obviously, $(\omega, S^{E}(x))$ has the same Lyapunov exponent as $(\omega, M^{E}(x))$.

\subsection{Anderson Localization for Analytic Quasi-Periodic Quantum Walks}

Recall that quantum walks can be viewed as a special class of extended CMV matrices, more specifically, those whose Verblunsky coefficients satisfy $\alpha_{2n}=0$.

According to the discussion above, denote the corresponding G-Z cocycle map and $n$-step transfer matrices by
\begin{equation*}
    S^{(E,\alpha)}(x)=\frac{1}{\rho(x)}\begin{pmatrix}
    E & -\overline{\alpha(x)} \\
     -\alpha(x)& E^{-1}
    \end{pmatrix}, \quad S_{n}^{(E,\alpha)}(x) = \prod_{j=n}^{1} S^{(E,\alpha)}(x+j\omega)
\end{equation*}
and the associated Lyapunov exponent of the quantum walk $U_{\omega}(x)$ is defined by
\begin{equation*}
    L(\omega,S^(E,\alpha))=\lim_{n \to \infty} \frac{1}{n}\int_{\T} \log \|S_{n}^{(E,\alpha)}(x)\| \, dx.
\end{equation*}
Next we give the statement of Anderson localization of quantum walks with analytic quasi-periodic coins.

\begin{theorem}
Consider the family of quantum walks $\{ U_{\omega}(x)\}_{\omega \in \T}$, where the coins are analytic. Let $\mathcal{I} \subset \T$ and $\mathcal{K} \subset \partial \mathbb{D}$ be compact intervals.

Assume that the Lyapunov exponent
\begin{equation*}
L(\omega,S^{(E,\alpha)}) > \delta_{0} > 0
\end{equation*}
for all $\omega \in \mathcal{I}$ and $E \in \mathcal{K}$.

Fix $x_{0} \in \T$. Then for almost every $\omega \in \mathcal{I}$, the corresponding quantum walk $U_{\omega}(x_{0})$ exhibits Anderson localization.
\end{theorem}

\begin{remark}
This result is an application of Theorem~\ref{talcmv}. For it to be applicable, we need to verify the uniform positivity assumption of $L(\omega,S^{(E,\alpha)})$ in a given setting. This means we need to prove that there exists an analytic function $\alpha(x):\T \to \D$ such that the corresponding G-Z cocycle has uniformly positive Lyapunov exponent. Such an example will be provided below.
\end{remark}

\subsection{Uniform Positivity of Lyapunov Exponents}

As discussed above, for analytic quasi-periodic Szeg\H o cocycles, explicit examples leading to uniformly positive Lyapunov exponents have already been found in \cite{zhangpositive}. Hence, we naturally attempt to modify the argument from \cite{zhangpositive} so that it applies to the G-Z cocycle associated with quantum walks. Indeed, as we will see in this subsection, the same function works due to the same reasoning.

\begin{lemma}{\label{qwpositivity}}
Let $\alpha(x), \theta(x), k$ and $\pi_{1}$ be the same as in Proposition~\ref{pzz1}. Then there exists a finite set $\mathcal{F} = \mathcal{F}(\theta,k)\subset \T\times\pD$ such that for any compact set $\mathcal{A} \subset (\T\times \pD) \backslash \mathcal{F}$,
\begin{equation*}
   L(\omega,S^{(E,\lambda v)}) \geq -\frac{1}{2}\log (1-\lambda) + c_{0}
\end{equation*}
for some $c_{0}=c_{0}(\theta,k,\mathcal{A})\in \mathbb{R}$ and all $(\omega,E;\lambda)\in \mathcal{A}\times(0,1)$. In particular, $L(\omega,S^{(E,\lambda v)})$ is uniformly positive in $(\omega,E) \in \mathcal{A}$ when $\lambda$ is sufficiently close to 1.
\end{lemma}

Before giving the proof, we first need some preparation.

\subsubsection{Polar decomposition}{\label{polardecomposition}}

For $A \in {\rm SL}(2,\mathbb{C})$, it is a standard result that it can be decomposed as $A = U_{1}\sqrt{A^{*}A}$, where $U_{1} \in \mathbb{SU}(2)$ and $\sqrt{A^{*}A}$ is a positive Hermitian matrix. Furthermore, we can decompose $\sqrt{A^{*}A}$ as $\sqrt{A^{*}A}=U_{2}\Lambda U^{*}_{2}$, where the column vectors of $U_{2}$ are eigenvectors of $\sqrt{A^{*}A}$ and $\Lambda = {\rm diag}(\|A\|,\|A\|^{-1})$. Hence, $U_{2}$ may be chosen to be in $ \mathbb{SU}(2)$. That is, any $ {\rm SL}(2,\mathbb{C})$ matrix $A$ can be decomposed as
\begin{equation*}
A = U_{1} U_{2} \Lambda U_{2}^{*}.
\end{equation*}

Now we consider a cocycle $(\omega, A)$, where $\omega \in \T$ and $A(x) \in {\rm SL}(2,\mathbb{C})$ is analytic. If $A(x) \in \mathbb{SU}(2)$, then there is no need to decompose it since $\sqrt{A^{*}A}$ is identity in this case. So let us assume further $A(x) \notin {\rm SU}(2)$ and decompose it as $A(x)=U_{1}(x)U_{2}(x)\Lambda(x)U_{2}^{*}(x)$. Let $U_{3}(x)=U_{1}(x-\omega)U_{2}(x-\omega) \in {\rm SU}(2)$, then we have
\begin{equation*}
U_{3}(x+\omega)^{*}A(x)U_{3}(x) = \Lambda(x) U(x),
\end{equation*}
where $U(x)=U_{2}(x)^{*}U_{1}(x-\omega)U_{2}(x-\omega) \in \mathbb{SU}(2)$. This is equivalent to
\begin{equation*}
(0,U_{3})^{-1} (\omega, A) (0,U_{3}) = (\omega, \Lambda U)
\end{equation*}
which means we can consider the cocycle $(\omega, \Lambda U)$ instead of $(\omega, A)$.

Let $U(x)=\begin{pmatrix} c(x) & -\overline{d(x)}  \\ d(x) & \overline{c(x)} \end{pmatrix}$, where $|c(x)|^{2}+|d(x)|^{2}=1$. Then the uniform positivity of Lyapunov exponents follows from \cite[Lemma 11]{zhangpositive}, if we can show  $\inf_{x \in \T} |c(x)| \geq \gamma \in (0,1)$.

\begin{prop}{\label{supositivity}}{\rm (}\cite[Lemma 11]{zhangpositive}{\rm )}
Let $(\omega, A)$ be a $\mathbb{SU}(2)$ free system, which means $A(x) \notin \mathbb{SU}(2)$, with the equivalent system $(\omega, \Lambda U)$, where $U$ is as above. Assume there exists a $0<\gamma<1$ such that $\inf_{x \in \T} |c(x)| \geq \gamma$ and let $\rho = \frac{1}{\gamma} + \sqrt{\frac{1}{\gamma^{2}}-1} > 1$. If
\begin{equation*}
\inf_{x \in \T} \|A(x)\| = \lambda > \rho,
\end{equation*}
then $(\omega, A) \in \mathcal{UH}$. Moreover, we have that
\begin{equation*}
L(\omega, A) \geq \ln \lambda - \ln 2\rho
\end{equation*}
for all $\lambda \in (0, \infty)$.
\end{prop}

Another important tool in the proof is {\it acceleration} of $\text{SL}(2,\mathbb{C})$ cocycles which is first defined in \cite{globaltheory},
\begin{equation*}
    \tilde{\omega}(\omega,A)= \lim_{y\to 0^{+}}\frac{1}{2\pi y}(L(\omega,A_{y})-L(\omega,A))
\end{equation*}
where $A \in C^{\omega}_{\delta}(\T, \text{SL}(2,\mathbb{C}))$ and $A_{y}(x)=A(x+iy) \in C^{\omega}(\T, \text{SL}(2,\mathbb{C}))$ for each $y \in (-\delta,\delta)$.

\begin{proof}[Proof of Lemma \ref{qwpositivity}]
Denote $\Omega_{\delta}=\{z=x+iy\in \mathbb{C}/\mathbb{Z}: |y| \in \delta \}$ the strip where the analytic cocycle $(\omega, S^{(E,\lambda v)}(x))$ can be extended.
Let $E=e^{2\pi it}$, $z=x+iy$ and $v(z)=r(z)e^{2 \pi i h(z)} = r e^{2\pi ih}$, where both $r$ and $h$ are real-valued functions.

By the formula $\|A\|^{2}=\frac{1}{2}(tr(A^{*}A)+\sqrt{tr(A^{*}A)-4})$, we can get
\begin{equation*}
\|S^{(E,\lambda v)}(x)\| = \sqrt{\frac{2 + \lambda^{2}(r^{2}+r^{-2})+\lambda(r+r^{-1}) \sqrt{4 + [\lambda (r - r^{-1})]^{2}}}{2(1 - \lambda^{2})}},
\end{equation*}
where $a=a(\lambda,r)=\lambda(r-r^{-1})+\sqrt{4+[\lambda(r-r^{-1})]^{2}}$.
Obviously, $r(z)$ and $a(z)$ are uniformly bounded away from $\infty$ and $0$ in any compact subregion of $\Omega_{\delta}$.

It's easy to see that $\|S^{(E,\lambda v)}(x)\| $ is uniformly of size $\sqrt{\frac{1}{1-\lambda}}$. In particular, if $y=0$, then $r=1$ and $\|S^{(E,\lambda v)}(x)\|=\sqrt{\frac{1+\lambda}{1-\lambda}}$ for all $x \in \T$. A direct calculation shows
\begin{equation*}
U_{2}(z)=\frac{1}{\sqrt{a^{2}+4}}\begin{pmatrix}
a(z) & \frac{2}{E}e^{-2\pi ih(z)} \\ -2Ee^{2\pi ih(z)} & a(z)
\end{pmatrix}.
\end{equation*}
Since
\begin{equation*}
U(z)=U_{2}(z)^{*}S^{(E,\lambda v)}(z-\omega)U_{2}(z-\omega)\Lambda(z-\omega)^{-1},
\end{equation*}
we obtain that the upper-left coefficient of $U$ is
\begin{equation*}
\begin{split}
c(z;\omega,E;\lambda)=\frac{c_{1}}{E}\{ a(z-\omega)&a(z)E^{2} + 2\lambda r(z-\omega)a(z-\omega)e^{2\pi i(h(z-\omega)-h(z))}\\ &+\frac{2\lambda a(z)E^{2}}{r(z-\omega)} + 4e^{2\pi i (h(z-\omega)-h(z))} \},
\end{split}
\end{equation*}
where
\begin{equation*}
c_{1}=\frac{\|S^{(E,\lambda v)}(z-\omega)\|^{-1}}{\sqrt{(a(z)^{2}+4)(a(z-\omega)^{2}+4)(1-\lambda^{2})}}
\end{equation*}
is uniformly bounded away from $\infty$ and $0$ for all $z$ in any compact subregion
of $\Omega_{\delta}$ and all $\lambda \in [0,1]$. Obviously, the uniform boundedness
of the left-upper entry of $U(z)$ depends on $a(z-\omega)a(z)E^{2} + 2\lambda
r(z-\omega)a(z-\omega)e^{2\pi i(h(z-\omega)-h(z))}\frac{2\lambda
a(z)E^{2}}{r(z-\omega)} + 4e^{2\pi i (h(z-\omega)-h(z))}$ which will be analyzed in the next step.

For $\lambda=1$, we have $a(1,r)=2r$ and thus
\begin{equation*}
\begin{split}
c(z;\omega,E;\lambda)&=\frac{4c_{1}e^{-2\pi ih(z)}}{E}\{ r(z-\omega)(v(z)E^{2}+v(z-\omega))+r(z-\omega)^{-1}(v(z)E^{2}+v(z-\omega)) \}\\
&=c_{2}[E^{2}v(z)+v(z-\omega)],
\end{split}
\end{equation*}
where $|c_{2}|$ and $|c_{2}^{-1}|$ are uniformly bounded. Then we can reduce the analysis of uniform positivity of $|c(z;\omega,E;1)|$ to $|g(z;\omega,E)|=E^{2}v(z)+v(z-\omega)$ and obtain
\begin{equation*}
E^{2}v(z)+v(z-\omega)=0 \Leftrightarrow \theta (z)-\theta (z-\omega) = \frac{1}{2}-2t-k\omega + m,
\end{equation*}
where $m \in \mathbb{Z}$. Since $q$ is the largest positive integer such that $\theta(z+\frac{1}{q})=\theta(z)$ and $\theta$ is a nonconstant real-analytic function, these together imply that:

(1) $\theta(z)-\theta(z-\omega)=\frac{1}{2}-2t-k\omega + m, \forall z \in \Omega_{\delta} $, only if $(\omega,t)=(\frac{p}{q},\frac{1}{4}+\frac{m}{2}-k\frac{p}{2q})$, where $p=0,1,\dots,q-1$. Obviously, such pairs $(\omega,t)$ form a finite set, which we denote by $\mathcal{F}$.

(2) $\theta(z)-\theta(z-\omega)=\frac{1}{2}-2t-k\omega + m$ has at most finitely many solutions in $\Omega_{\delta}$.

\medskip

Compared with the proof of \cite[Theorem A]{zhangpositive}, we can find the finite set here is different. The following part of the proof is the same with that in \cite[Theorem A]{zhangpositive}.

Now we take $\mathcal{A} \subset (\T\times \pD) \setminus \mathcal{F}$, which means for each $(\omega_{j},e^{2\pi i t_{j}}) \in \mathcal{A}$, it is not in $\mathcal{F}$ and hence satisfies condition (2). So for any $(\omega_{j},e^{2\pi i t_{j}}) \in \mathcal{A}$, we can find some height $y_{j}$ such that $E^{2}v(z)+v(z-\omega)$ is bounded away from zero for all $x\in \T$, which means $|c(x+iy_{j};\omega_{j},t_{j};1)|$ has the same property. Then for each $(\omega_{j},e^{2\pi i t_{j}}) \in \mathcal{A}$, we can find some connected open set $\mathcal{O}_{j}$ satisfying $(\omega_{j},e^{2\pi i t_{j}}) \in \mathcal{O}_{j} \subset (\T\times \pD) \setminus \mathcal{F}$ and some large $\lambda_{j}>0$ such that $|c(x+iy_{j};\alpha,t;1)|$ is bounded away from zero for all $(x,\omega,e^{2\pi it},\lambda) \in \T \times \mathcal{O}_{j} \times [\lambda_{j},1]$.

Now by Proposition \ref{supositivity}, without loss of generality, we can assume $\lambda_{j}$ is such that there is a constant $\eta_{j}=\eta_{j}(\lambda_{j},\mathcal{O}_{j})$ and for each $(\omega,E,\lambda)\in \mathcal{O}_{j}\times[\lambda_{j},1)$, the following hold:

(1) $(\omega,S_{y_{j}}^{(E,\lambda v)})\in \mathcal{UH}$;

(2) $L(\omega,S_{y_{j}}^{(E,\lambda v)})\geq-\frac{1}{2}\ln (1-\lambda)+\eta_{j}$.

\medskip

By compactness of $\mathcal{A}$, there exist finitely many $j$, say $j=1,2,\dots ,l$, such that $\mathcal{A} \subset \cup_{1\leq j\leq l}\mathcal{O}_{j}$. Let $\lambda_{0}=\max \{\lambda_{1},\dots,\lambda_{l}\}$, then by \cite[Theorem 9]{zhangpositive}, there are at most $l$ nonnegative integers $n_{j}$, such that the acceleration $\tilde{\omega}(\omega,S_{y_{j}}^{(E,\lambda v)})=n_{j}$ on $\mathcal{O}_{j}\times[\lambda_{0}, 1)$. Here $n_{j} \geq 0$ follows from the fact that $L(\omega,S_{y})$ is convex in $y$ and $S^{(E,\lambda v)} \in \mathbb{SU}(1, 1)$ when $y=0$.

Let $n_{0}=\max\{n_{1},\dots,n_{l}\}$ and $y_{0}=\max\{y_{1},\dots,y_{l}\}$. By the definition of acceleration, for all $(\omega,E,\lambda) \in \mathcal{A}\times [\lambda_{0},1) $, there exists $j$, $1\leq j\leq l$ such that
\begin{equation*}
L(\omega,S^{(E,\lambda v)}) \geq L(\omega,S_{y_{j}}^{(E,\lambda v)})-\frac{n_{0}y_{0}}{2\pi} \geq -\frac{1}{2}\ln (1-\lambda) + \eta_{j} -\frac{n_{0}y_{0}}{2\pi}
\end{equation*}
Thus for all $(\omega,E,\lambda)\in \mathcal{A}\times[\lambda_{0},1)$,
\begin{equation*}
L(\omega,S^{(E,\lambda v)}) \geq -\frac{1}{2}\ln (1-\lambda) + c_{0}
\end{equation*}
with $c_{0}=\min_{1\leq j \leq l}{\eta_{j} -\frac{n_{0}y_{0}}{2\pi}}$. Obviously, we can change $c_{0}$ to allow
\begin{equation*}
L(\omega,S^{(E,\lambda v)}) \geq -\frac{1}{2}\ln (1-\lambda) + c_{0}
\end{equation*}
for all $(\omega,E,\lambda)\in \mathcal{A}\times(0,1)$.

This implies for $\lambda$ sufficiently close to 1, we can have the uniform positivity of $L(\omega,S^{(E,\lambda v)})$.

\end{proof}

\section{Appendix}

\subsection{Restriction of Eigenequations}{\label{boundary}}

In the Schr\"odinger case, by restricting the eigenequation $(H(\omega, x)-E)\xi = 0$ to a finite interval $\Lambda$, we get exactly two boundary terms, which in turn yields the identity $\xi(n) = - G^{E}_{\Lambda}(n,a)\xi(a-1)-G^{E}_{\Lambda}(n,b)\xi(b+1)$.

But in the CMV case, the analog of this formula depends on the parity of the endpoints of the finite interval. Concretely, if $\xi$ is a solution of the difference equation $\CE \xi = z\xi$, define
\[
\widetilde\psi(a)
=
\begin{cases}
\left(z\bar{\beta}-\alpha_a\right)\xi(a) - \rho_a \xi(a+1)
& a \text{ is even,} \\
\left(z\alpha_a - \beta\right)\xi(a) + z\rho_a \xi(a+1)
& a \text{ is odd,}
\end{cases}
\]
and
\[
\widetilde\psi(b)
=
\begin{cases}
\left(z\bar{\gamma}-\alpha_b\right)\xi(b) - \rho_b \xi(b-1)
& b \text{ is even,} \\
\left(z\alpha_b - \gamma\right)\xi(b) + z\rho_{b-1} \xi(b-1)
& b \text{ is odd.}
\end{cases}
\]
Then, for $a < n < b$ we have
\begin{equation}\label{eq:CMV:deveGreen}
\xi(n)
=
G_{[a,b]}^{\beta,\gamma}(n,a;z)\widetilde\psi(a)
+ G_{[a,b]}^{\beta,\gamma}(n,b;z)\widetilde\psi(b).
\end{equation}

\subsection{Paving Property}{\label{paving}}

For $\Lambda \subset \mathbb{Z}$, let
\begin{equation*}
    \mathcal{E}_{\Lambda} = R_{\Lambda} \mathcal{E} R_{\Lambda}^*,
\end{equation*}
where $R_{\Lambda}$ is the restriction operator, and
\begin{equation*}
    G_{\Lambda} = (z - \mathcal{E}_{\Lambda})^{-1}.
\end{equation*}
Recall the resolvent identity,
\begin{equation*}
    G_{\Lambda} = (G_{\Lambda_{1}}+G_{\Lambda_{2}})-(G_{\Lambda_{1}} + G_{\Lambda_{2}})(\mathcal{E}_{\Lambda}-\mathcal{E}_{\Lambda_{1}}-\mathcal{E}_{\Lambda_{2}})G_{\Lambda},
\end{equation*}
where $\Lambda \subset \mathbb{Z}$ is a disjoint union $\Lambda = \Lambda_{1} \cup \Lambda_{2}$, provided the inverses make sense. One of the consequences of the resolvent identity is the following \emph{paving property}.

Let $I \subset \mathbb{Z}$ be an interval of size $N > n$, such that for each $x \in I$, there is a size $n$ interval $I' \subset I$ satisfying
\begin{equation*}
    \Big \{ y \in I : |x-y| < \frac{n}{10} \Big \} \subset I'
\end{equation*}
and
\begin{equation*}
    \lvert G_{I'}(n_{1}, n_{2}) \rvert < e^{-c \lvert n_{1} - n_{2} \rvert }
\end{equation*}
for some constant $c > 0$ and $n$ sufficiently large. Then also
\begin{equation*}
    \lvert G_{I}(n_{1}, n_{2}) \rvert < e^{-\frac{c}{2} \lvert n_{1}- n_{2} \rvert }.
\end{equation*}
For the proof of this statement, see part (\uppercase\expandafter{\romannumeral4}) of \cite{BG00}.

\section{Acknowledgements}

We would like to thank Zhenghe Zhang for numerous useful discussions and suggestions throughout the preparation of this paper. F.~W.\ would like to thank his advisor Daxiong Piao (Professor at Ocean University of China) for his support.

\end{document}